\documentclass{article}
\usepackage{array,amsmath,amsthm}
\numberwithin{equation}{section}
\usepackage{amssymb}
\usepackage{indentfirst}
\usepackage{geometry}
\begin{document}

\newtheorem{thm}{Theorem}[section]
\newtheorem{cor}[thm]{Corollary}
\newtheorem{prop}[thm]{Proposition}
\newtheorem{conj}[thm]{Conjecture}
\newtheorem{lem}[thm]{Lemma}
\newtheorem{Def}[thm]{Definition}
\newtheorem{rem}[thm]{Remark}
\newtheorem{prob}[thm]{Problem}
\newtheorem{ex}{Example}[section]

\newcommand{\be}{\begin{equation}}
\newcommand{\ee}{\end{equation}}
\newcommand{\ben}{\begin{enumerate}}
\newcommand{\een}{\end{enumerate}}
\newcommand{\beq}{\begin{eqnarray}}
\newcommand{\eeq}{\end{eqnarray}}
\newcommand{\beqn}{\begin{eqnarray*}}
\newcommand{\eeqn}{\end{eqnarray*}}
\newcommand{\bei}{\begin{itemize}}
\newcommand{\eei}{\end{itemize}}

\newcommand{\pa}{{\partial}}
\newcommand{\V}{{\rm V}}
\newcommand{\R}{{\bf R}}
\newcommand{\K}{{\rm K}}
\newcommand{\e}{{\epsilon}}
\newcommand{\tomega}{\tilde{\omega}}
\newcommand{\tOmega}{\tilde{Omega}}
\newcommand{\tR}{\tilde{R}}
\newcommand{\tB}{\tilde{B}}
\newcommand{\tGamma}{\tilde{\Gamma}}
\newcommand{\fa}{f_{\alpha}}
\newcommand{\fb}{f_{\beta}}
\newcommand{\faa}{f_{\alpha\alpha}}
\newcommand{\faaa}{f_{\alpha\alpha\alpha}}
\newcommand{\fab}{f_{\alpha\beta}}
\newcommand{\fabb}{f_{\alpha\beta\beta}}
\newcommand{\fbb}{f_{\beta\beta}}
\newcommand{\fbbb}{f_{\beta\beta\beta}}
\newcommand{\faab}{f_{\alpha\alpha\beta}}

\newcommand{\pxi}{ {\pa \over \pa x^i}}
\newcommand{\pxj}{ {\pa \over \pa x^j}}
\newcommand{\pxk}{ {\pa \over \pa x^k}}
\newcommand{\pyi}{ {\pa \over \pa y^i}}
\newcommand{\pyj}{ {\pa \over \pa y^j}}
\newcommand{\pyk}{ {\pa \over \pa y^k}}
\newcommand{\dxi}{{\delta \over \delta x^i}}
\newcommand{\dxj}{{\delta \over \delta x^j}}
\newcommand{\dxk}{{\delta \over \delta x^k}}

\newcommand{\px}{{\pa \over \pa x}}
\newcommand{\py}{{\pa \over \pa y}}
\newcommand{\pt}{{\pa \over \pa t}}
\newcommand{\ps}{{\pa \over \pa s}}
\newcommand{\pvi}{{\pa \over \pa v^i}}
\newcommand{\ty}{\tilde{y}}
\newcommand{\bGamma}{\bar{\Gamma}}

\title {Concentration and relevant properties of Finsler metric measure manifolds\footnote{This work was supported by National Natural Science Foundation of China (Grant No. 12371051, 12141101)} }
\author{Xinyue Cheng and Yalu Feng\\
School of Mathematical Sciences, Chongqing Normal University, \\ Chongqing, 401331, P.R. China\\
E-mail: chengxy@cqnu.edu.cn and fengyl2824@qq.com}

\date{}

\maketitle

\begin{abstract}
In this paper, we study  systematically the concentration properties of Finsler metric measure manifolds. We establish the relationships between the concentration properties and  the observable diameter, isoperimetric inequalities and the first eigenvalue. In particular, as an application, we derive a Cheng type upper bound estimate for the first closed eigenvalue via the concentration property. The researches in this paper enrich and extend the concentration theory in Finsler geometry, even in irreversible metric measure spaces. \\
{\bf Keywords:} irreversible metric measure space; Finsler metric measure manifold; concentration; observable diameter; isoperimetric inequality; the first eigenvalue\\
{\bf Mathematics Subject Classification:} 53C60,  53B40, 58C35
\end{abstract}

\section{Introduction}\label{intr}

Let $\mathbb{R}^{\infty}$ be an infinite-dimensional Hilbert space and $S^{\infty}$ the unit sphere in $\mathbb{R}^{\infty}$. Consider a finite-dimensional unit sphere $S$ in $\mathbb{R}^{\infty}$, which is the intersection of $S^{\infty}$ with a finite-dimensional subspace in $\mathbb{R}^{\infty}$. The diameter of sphere $S$ is always equal to $\pi$ and the sectional curvature of $S$ is always equal to one. A natural question arises: How can one detect the dimension of $S$ and how to tell the difference between two finite-dimensional unit spheres in $\mathbb{R}^{\infty}$? A satisfactory answer to this question is provided by the  L\'{e}vy concentration theory of metric measure spaces \cite{Ly,SHEN}. The L\'{e}vy concentration theory of metric measure spaces has been developed further by Gromov-Milman \cite{Gr}, V.D. Milman \cite{Mil1, Mil2}, M. Talagrand \cite{Tal1, Tal2}, M. Gromov \cite{GROMOV}, M. Ledoux \cite{Le}, E. Milman \cite{EMil} and other geometers. The theory has many applications in various areas of mathematics, such as geometry, analysis, probability theory and discrete mathematics \cite{Le, MILSC}. As mentioned by M. Gromov, the concentration of measure was an elementary yet non-trivial observation \cite{Le}.

Finsler geometry provides a natural mathematical framework for describing a wide range of anisotropic phenomena. It is natural to extend the theory of concentration of measure to this geometric setting. In 1999, Gu-Shen \cite{GuShen} studied the expansion distance and the observable diameter in Finsler metric measure spaces, which were first introduced by M. Gromov \cite{GROMOV}. They obtained the upper bound estimates controlled by the first eigenvalue or the isoperimetric profile for these quantities on closed Finsler metric measure manifolds. They also derived a lower bound estimate for the expansion distance on compact reversible Finsler metric measure spaces whose Ricci curvature has a negative lower bound (see also \cite{SHEN}). Later, S. Ohta \cite{OHTA} derived a normal concentration property of a forward or backward complete Finsler metric measure space $(M, F, \mathfrak{m})$ by using a Talagrand-type inequality (also called the transportation cost inequality) under the conditions that $\mathfrak{m}(M)=1$ and the weighted Ricci curvature ${\rm Ric}_{\infty}$ has a positive lower bound.

In this paper, inspired by the works of Gromov-Milman \cite{Gr}, M. Ledoux \cite{Le} and Z. Shen \cite{SHEN}, we investigate systematically the concentration properties of Finsler metric measure manifolds. Our discussions involve the concentration function, the observable diameter, isoperimetric inequalities, the first eigenvalue and their relations.

The paper is organized as follows. Section \ref{pre} introduces some necessary definitions and notations on irreversible metric spaces and Finsler manifolds. In Section \ref{con}, we first study some important concentration properties of irreversible metric measure spaces, including normal concentration, exponential concentration, and $(p,q)-$Moment concentration. Then, as the special setting, we obtain the corresponding  concentration properties of Finsler metric measure manifolds. Section \ref{diacon} explores the relationship between observable diameter and the concentration function. Further, we derive a normal concentration property from an isoperimetric inequality in Section \ref{Iso}. Finally, in Section \ref{Spec},  we prove an exponential concentration property via the first eigenvalue and further derive a Cheng type upper bound for the first closed eigenvalue by concentration properties.

\section{Preliminaries}\label{pre}

In this section, we briefly review some essential definitions and notations in Finsler geometry.  It is worth to note that a Finsler space is not a metric measure space in the usual sense since the Finsler metric $F$ may be non-reversible, that is, $F(x, y)\neq F(x, -y)$ may occur for some $(x,y) \in TM$.
This non-reversibility causes the asymmetry of the associated distance function. Therefore, we begin by introducing the following irreversible metric spaces.

\subsection{Irreversible metric spaces}

\begin{Def}{\rm (\cite{KZ})}
Let $X$ be a set and $d: X \times X \rightarrow[0, \infty)$ be a function on $X$. The pair $(X, d)$ is called an irreversible metric space if for any $x_{1}, x_{2}, x_{3} \in X$:
\begin{itemize}
  \item [(i)] $d(x_{1}, x_{2}) \geq 0$, with equality if and only if $x_{1}= x_{2}$;
  \item [(ii)] $d(x_{1}, x_{3}) \leq d(x_{1}, x_{2})+d(x_{2}, x_{3})$.
\end{itemize}
In particular, if $(X, d)$ is an irreversible metric space, then $d$ is called a metric on $X$.

In addition, if for every $x_{1}, x_{2}\in X$, one has
\begin{itemize}
  \item [(iii)] $d(x_{1}, x_{2})=d(x_{2},x_{1})$,
\end{itemize}
the pair $(X, d)$ is called a reversible metric space.
\end{Def}

Since the metric $d$ of an irreversible metric space $(X, d)$ could be asymmetric, there are two kinds of balls, that is, forward and backward balls, respectively. More precisely, given any $r>0$ and a point $x \in X$, the forward ball $B_x^{+}(r)$ (resp., backward ball $B_x^{-}(r)$) of radius $r$ centered at $x$ is defined as
$$
B_x^{+}(r):=\{z \in X \mid d(x, z)<r\}, \quad B_x^{-}(r):=\{z \in X \mid d(z, x)<r\}.
$$

\subsection{Finsler manifolds}
Let $M$ be an $n$-dimensional smooth manifold and  $TM$ be the tangent bundle on $M$. Denote the elements in $TM$ by $(x, y)$ with $y \in T_{x} M$. Let $T M_{0}:=TM \backslash\{0\}$ and $\pi: T M_{0} \rightarrow M$ be the natural projective map. The pull-back $\pi^{*} T M$ is an $n$-dimensional vector bundle on $T M_0$. A Finsler metric on manifold $M$ is a function $F: T M \longrightarrow[0, \infty)$  satisfying the following properties:
\begin{itemize}
  \item [(1)] (Regularity) $F$ is $C^{\infty}$ on $TM_{0}$;
  \item [(2)] (Positive 1-homogeneity) $F(x,\lambda y)=\lambda F(x,y)$ for any $(x,y)\in TM$ and all $\lambda >0$;
  \item [(3)] (Strong convexity) $F$ is strongly convex, that is, the matrix $\left(g_{ij}(x,y)\right)=\left(\frac{1}{2}(F^{2})_{y^{i}y^{j}}\right)$ is positive definite for any nonzero $y\in T_{x}M$.
\end{itemize}
Such a pair $(M,F)$ is called a Finsler manifold and $g:=g_{ij}(x,y)dx^{i}\otimes dx^{j}$ is called the fundamental tensor of $F$. For a non-vanishing vector field $V$ on $M$, one introduces the weighted Riemannian metric $g_V$ on $M$ given by
$$
g_V(y, w)=g_{ij}(x, V_x)y^i w^j
$$
for $y,\, w\in T_{x}M$. In particular, $g_{V}(V,V)=F^{2}(x, V)$ for any $V \in T_{x}M$.

We define the reverse metric $\overleftarrow{F} $ of $F$ by $\overleftarrow{F}(x, y):=F(x,-y)$ for all $(x, y) \in T M$. It is easy to see that $\overleftarrow{F}$ is also a Finsler metric on $M$. A Finsler metric $F$ on $M$ is said to be reversible if $\overleftarrow{F}(x, y)=F(x, y)$ for all $(x, y) \in T M$. Otherwise, we say $F$ is irreversible.

Let $(M,F)$ be a Finsler manifold of dimension $n$. The pull-back $\pi ^{*}TM$ admits a unique linear connection, which is called the Chern connection. The Chern connection $D$ is determined by the following equations  \cite{BaoChern}
\beq
&& D^{V}_{X}Y-D^{V}_{Y}X=[X,Y], \label{chern1}\\
&& Zg_{V}(X,Y)=g_{V}(D^{V}_{Z}X,Y)+g_{V}(X,D^{V}_{Z}Y)+ 2C_{V}(D^{V}_{Z}V,X,Y) \label{chern2}
\eeq
for $V\in TM\backslash\{0\}$  and $X, Y, Z \in TM$, where
$$
C_{V}(X,Y,Z):=C_{ijk}(x,V)X^{i}Y^{j}Z^{k}=\frac{1}{4}\frac{\pa ^{3}F^{2}(x,V)}{\pa V^{i}\pa V^{j}\pa V^{k}}X^{i}Y^{j}Z^{k}
$$
is the Cartan tensor of $F$ and $D^{V}_{X}Y$ is the covariant derivative with respect to the reference vector $V$.

Given a non-vanishing vector field $V$ on $M$,  the Riemannian curvature  $R^V$ is defined by
$$
R^V(X, Y) Z=D_X^V D_Y^V Z-D_Y^V D_X^V Z-D_{[X, Y]}^V Z
$$
for any vector fields $X$, $Y$, $Z$ on $M$. For two linearly independent vectors $V, W \in T_x M \backslash\{0\}$, the flag curvature is defined by
$$
\mathcal{K}^V(V, W)=\frac{g_V\left(R^V(V, W) W, V\right)}{g_V(V, V) g_V(W, W)-g_V(V, W)^2}.
$$
Then the Ricci curvature is defined as
$$
\operatorname{Ric}(V):=F(x, V)^{2} \sum_{i=1}^{n-1} \mathcal{K}^V\left(V, e_i\right),
$$
where $e_1, \ldots, e_{n-1}, \frac{V}{F(V)}$ form an orthonormal basis of $T_x M$ with respect to $g_V$.

\subsection{Finsler Laplacian}

We denote a Finsler manifold $(M, F)$ equipped with a positive smooth measure $\mathfrak{m}$ by $(M, F, \mathfrak{m})$ which we call a Finsler metric measure manifold.

Given a Finsler structure $F$ on $M$,  there is a Finsler co-metric $F^{*}$ on $M$ which is non-negative function on the cotangent bundle $T^{*}M$ given by
\be
F^{*}(x, \xi):=\sup\limits_{y\in T_{x}M\setminus \{0\}} \frac{\xi (y)}{F(x,y)}, \ \ \forall \xi \in T^{*}_{x}M. \label{co-Finsler}
\ee
We call $F^{*}$ the dual Finsler metric of $F$.  For any vector $y\in T_{x}M\backslash\{0\}$, $x\in M$, the covector $\xi =g_{y}(y, \cdot)\in T^{*}_{x}M$ satisfies
\be
F(x,y)=F^{*}(x, \xi)=\frac{\xi (y)}{F(x,y)}. \label{shenF311}
\ee
Conversely, for any covector $\xi \in T_{x}^{*}M\backslash\{0\}$, there exists a unique vector $y\in T_{x}M\backslash\{0\}$ such that $\xi =g_{y}(y, \cdot)\in T^{*}_{x}M$ (Lemma 3.1.1, \cite{SHEN}). Naturally,  we define a map ${\cal L}: TM \rightarrow T^{*}M$ by
$$
{\cal L}(y):=\left\{
\begin{array}{ll}
g_{y}(y, \cdot), & y\neq 0, \\
0, & y=0.
\end{array} \right.
$$
It follows from (\ref{shenF311}) that
$$
F(x,y)=F^{*}(x, {\cal L}(y)).
$$
Thus ${\cal L}$ is a norm-preserving transformation. We call ${\cal L}$ the Legendre transformation on Finsler manifold $(M, F)$.

Let
$$
g^{*kl}(x,\xi):=\frac{1}{2}\left[F^{*2}\right]_{\xi _{k}\xi_{l}}(x,\xi).
$$
For any $\xi ={\cal L}(y)$, we have
$$
g^{*kl}(x,\xi)=g^{kl}(x,y),
$$
where $\left(g^{kl}(x,y)\right)= \left(g_{kl}(x,y)\right)^{-1}$.

Given a smooth function $u$ on $M$, the differential $d u_x$ at any point $x \in M$, $d u_x=\frac{\partial u}{\partial x^i}(x) d x^i$ is a linear function on $T_x M$. We define the gradient vector $\nabla u(x)$ of $u$ at $x \in M$ by $\nabla u(x):=\mathcal{L}^{-1}(d u(x)) \in T_x M$. In a local coordinate system, we can express $\nabla u$ as
\be \label{nabna}
\nabla u(x)= \begin{cases}g^{* i j}(x, d u) \frac{\partial u}{\partial x^i} \frac{\partial}{\partial x^j}, & x \in M_u, \\ 0, & x \in M \backslash M_u,\end{cases}
\ee
where $M_{u}:=\{x \in M \mid d u(x) \neq 0\}$ \cite{SHEN}. In general, $\nabla u$ is only continuous on $M$, but smooth on $M_{u}$. Further,  for $u \in W^{1,2}(M)$, noting that $\nabla u$ is weakly differentiable, the Finsler Laplacian  $\Delta u$ is defined by
$$
\int_M \phi \Delta u d \mathfrak{m}:=-\int_{M} d \phi(\nabla u) d\mathfrak{m}
$$
for $\phi \in \mathcal{C}_{0}^{\infty}(M)$, where ${\cal C}^{\infty}_{0} (M)$ denotes the set of all smooth compactly supported functions on $M$  \cite{SHEN}.

\subsection{Weighted Ricci curvature}
Let $(M, F, \mathfrak{m})$ be  an $n$-dimensional Finsler manifold $(M, F)$ equipped with a positive smooth measure $\mathfrak{m}$. Write the volume form $d \mathfrak{m}$ of  $\mathfrak{m}$ as $d \mathfrak{m} = \sigma(x) dx^{1} dx^{2} \cdots d x^{n}$. Define
\be\label{Dis}
\tau (x, y):=\ln \frac{\sqrt{{\rm det}\left(g_{i j}(x, y)\right)}}{\sigma(x)}.
\ee
We call $\tau$ the distortion of $F$. It is natural to study the rate of change of the distortion along geodesics. For a vector $y \in T_{x} M \backslash\{0\}$, let $ \gamma = \gamma(t)$ be the geodesic with $\gamma (0)=x$ and $\dot{\gamma}(0)=y.$  Set
\be
{\bf S}(x, y):= \frac{d}{d t}\left[\tau(\gamma (t), \dot{\gamma}(t))\right]|_{t=0}.
\ee
$\mathbf{S}$ is called the S-curvature of $F$ \cite{ChernShen}. Let $Y$ be a $C^{\infty}$ geodesic field on an open subset $U \subset M$ (i.e. all integral curves of $Y$ on $U$ are geodesics) and $\hat{g}=g_{Y}$ .  Let
\be
d \mathfrak{m}:=e^{- \psi} {\rm Vol}_{\hat{g}}, \ \ \ {\rm Vol}_{\hat{g}}= \sqrt{{det}\left(g_{i j}\left(x, Y_{x}\right)\right)}dx^{1} \cdots dx^{n}. \label{voldecom}
\ee
It is easy to see that $\psi$ is given by
$$
\psi (x)= \ln \frac{\sqrt{\operatorname{det}\left(g_{i j}\left(x, Y_{x}\right)\right)}}{\sigma(x)}=\tau\left(x, Y_{x}\right),
$$
which is just the distortion of $F$ along $Y_{x}$ at $x\in M$ \cite{ChernShen, SHEN}. Let $y := Y_{x}\in T_{x}M$ (that is, $Y$ is a geodesic extension of $y\in T_{x}M$). Then, by the definitions of the S-curvature, we have
\beqn
&&  {\bf S}(x, y)= Y[\tau(x, Y)]|_{x} = d \psi (y),  \\
&&  \dot{\bf S}(x, y)= Y[{\bf S}(x, Y)]|_{x} =y[Y(\psi)],
\eeqn
where $\dot{\bf S}(x, y):={\bf S}_{|m}(x, y)y^{m}$ and ``$|$" denotes the horizontal covariant derivative with respect to the Chern connection. Further, Ohta introduces the weighted Ricci curvatures in Finsler geometry in \cite{Ohta0} which are defined as follows \cite{Ohta0,OHTA}
\beq
{\rm Ric}_{N}(y)&=& {\rm Ric}(y)+ \dot{\bf S}(x, y) -\frac{{\bf S}(x, y)^{2}}{N-n},   \label{weRicci3}\\
{\rm Ric}_{\infty}(y)&=& {\rm Ric}(y)+ \dot{\bf S}(x, y). \label{weRicciinf}
\eeq
We say that Ric$_{N}\geq K$ for some $K\in \mathbb{R}$ if ${\rm Ric}_{N}(v)\geq KF^{2}(v)$ for all $v\in TM$, where $N\in \mathbb{R}\backslash\{n\}$ or $N= \infty$.

\subsection{Irreversible metric induced by a Finsler metric}

For $x_1, x_2 \in M$, the distance from $x_1$ to $x_2$ is defined by
$$
d_{F}\left(x_{1}, x_{2}\right):=\inf _\gamma \int_{0}^{1} F(\gamma (t), \dot{\gamma}(t)) d t,
$$
where the infimum is taken over all $C^1$ curves $\gamma:[0,1] \rightarrow M$ such that $\gamma(0)=$ $x_1$ and $\gamma(1)=x_2$.
Thus, $d_{F}: M\times M\rightarrow [0, \infty)$ is a continuous function with
\begin{itemize}
  \item [(1)] $d_{F}(x_{1}, x_{2})\geq 0$, with equality if and only if $x_{1}=x_{2}$;
  \item [(2)] $d_{F}(x_{1}, x_{2})\leq d_{F}(x_{1}, p)+d_{F}(p, x_{2})$, $\forall p\in M$.
\end{itemize}
However, $d_{F} \left(x_1, x_2\right) \neq d_{F} \left(x_2, x_1\right)$ unless $F$ is reversible.
We can define the forward and backward geodesic balls of radius $r$ with center at $x_{0}\in M$, respectively, as
$$
B_{x_0}^{+}(r):=\{x \in M \mid d_{F}(x_{0}, x)<r\},\qquad B_{x_{0}}^{-}(r):=\{x \in M \mid d_{F}(x, x_{0})<r\}.
$$

A $C^{\infty}$-curve $\gamma:[0,1] \rightarrow M$ is called a geodesic  if $F(\gamma, \dot{\gamma})$ is constant and it is locally minimizing.
The exponential map $\exp _x: T_x M \rightarrow M$ is defined by $\exp _x(v)=\gamma(1)$ for $v \in T_x M$ if there is a geodesic $\gamma:[0,1] \rightarrow M$ with $\gamma(0)=x$ and $\dot{\gamma}(0)=v$.  A Finsler manifold $(M, F)$ is said to be forward complete (resp. backward complete) if each geodesic defined on $[0, \ell)$ (resp. $(-\ell, 0])$ can be extended to a geodesic defined on $[0, \infty)$ (resp. $(-\infty, 0])$.  We say $(M, F)$ is complete if it is both forward complete and backward complete.

\section{Concentration of irreversible metric measure spaces}\label{con}

A triple $(X, d, \mu)$ is called an irreversible metric measure space if $(X, d)$ is an irreversible metric space endowed with a Borel probability measure $\mu$. The concentration function $\alpha_{(X, d, \mu)}(r)$ ($\alpha_{\mu}(r)$ for short) is defined by
\begin{equation}\label{confun}
\alpha_{(X, d, \mu)}(r):=\sup\limits_{A}\left\{1-\min\{\mu\left(B^{+}(A, r)\right), \mu\left(B^{-}(A, r)\right)\} \mid A\subset X, \mu(A)\geq \frac{1}{2}\right\}
\end{equation}
for $ r>0$, where $A$ is a Borel set, $B^{+}(A, r):=\{x\in X \mid \inf\limits_{z\in A}d(z, x)<r\}$ and $B^{-}(A, r):=\{x\in X \mid \inf\limits_{z\in A}d(x, z)<r\}$ are the forward and backward open $r$-neighborhood of $A$, respectively. It is worth noting that, due to the asymmetry of the metric $d$, the concentration function $\alpha_{\mu}(r)$ we define here is slightly different from the definitions given by M. Ledoux \cite{Le} and S. Ohta \cite{OHTA}. From the definition of the concentration function $\alpha_{\mu}(r)$, we observe that $\alpha_{\mu}(r)$ is less than or equal to $\frac{1}{2}$ and decreases to $0$ as $r\rightarrow +\infty$.
In particular, we say that $\mu$ has normal concentration on $(X, d)$ if there are positive constants $C$ and $c$ such that, for every $r>0$,
\begin{equation}\label{normal}
\alpha_{(X, d, \mu)}(r)\leq C {\rm e}^{-cr^2}.
\end{equation}
We also say that $\mu$ has exponential concentration if there are positive constants $C$ and $c$ such that, for every $r>0$,
\begin{equation}\label{exp}
\alpha_{(X, d, \mu)}(r)\leq C {\rm e}^{-cr}.
\end{equation}

In the following, we will describe the concentration properties of the Lipschitz function $f$ around the median value $m_{f}$.
\begin{Def}{\rm (\cite{KZ})}
Let $(X, d)$ be an irreversible metric space. A function $f$: $X\rightarrow \mathbb{R}$ is said to be a $L$-Lipschitz function if there exists a constant $L>0$ such that
$$
f(z)-f(x)\leq L d(x,z), \ \ \ \ \text{for all} \ x, z\in X.
$$
We denote by ${\rm Lip}_{L}(X)$ the set of all $L$-Lipschitz functions on $X$.
\end{Def}

If $\mu$ is a probability measure on the Borel sets of $(X, d)$ and $f$ is a measurable real-valued function on $(X, d)$, we say that $m_{f}$ is a median of $f$ for $\mu$ if
$$
\mu(\{x\in X \mid f(x)\leq m_{f}\})\geq \frac{1}{2} \ \ \text{and} \ \ \mu(\{x\in X \mid f(x)\geq m_{f}\})\geq \frac{1}{2}.
$$

Let $f\in{\rm Lip}_{L}(X)$ and $A:=\{x\in X|f(x)\leq m_{f}\}$. Then we have $f(z)\leq m_{f}+L d(x, z)$ for any $x\in A$ and $z\in X$. If $z\in B^{+}(A, r)$, then $f(z)< m_{f}+Lr$. Thus $\mu(B^{+}(A, r))\leq \mu(\{z\in X \mid f(z)<m_{f}+Lr\})$. Since $\mu (A)\geq \frac{1}{2}$, by the definition of the concentration function $\alpha_{\mu}(r)$, we have
\beq
\mu(\{z\in X \mid f(z)\geq m_{f}+r\})\leq \alpha_{\mu}\left(r/L\right).\label{mf1}
\eeq
Similarly, let $A':=\{x\in X \mid f(x)\geq m_{f}\}$.  We can get $f(z)\geq f(x)-L d(z, x)\geq m_{f}-Ld(z, x)$ for any $x\in A'$ and $z\in X$. If $z\in B^{-}(A', r)$, then $f(z)> m_{f}-Lr$. Thus we have
\be
\mu(\{z\in X \mid f(z)\leq m_{f}-r\})\leq \alpha_{\mu}\left(r/L\right). \label{mf2}
\ee
Therefore, together (\ref{mf2}) with (\ref{mf1}), we obtain that for every $r>0$,
\beq
\mu(\{z\in X \mid |f(z)- m_{f}|\geq r\})\leq 2\alpha_{\mu}\left(r/L\right).\label{mf3}
\eeq

\begin{prop}\label{f-mean}
Let $(X, d, \mu)$ be an irreversible metric measure space and $\beta(r)$ be a non-negative decreasing function on $(0, \infty)$.
\bei
\item[{\rm (1)}] If $\mu\left(\{x\in X| \ |f(x)-\int_{X} f d\mu|\geq r\}\right)\leq \beta(r)$ for any 1-Lipschitz function $f$ and any $r>0$, then
\be
1-\mu(B^{+}(A, r))\leq \beta (\mu(A) r) \ \ \text{and} \ \ 1-\mu(B^{-}(A, r))\leq \beta (\mu(A) r) \label{3.2-1}
\ee
for every Borel set $A$ with $\mu(A)>0$. In particular, $\alpha_{\mu}(r)\leq \beta(r/2)$ for every $r>0$.

\item[{\rm (2)}]  If $\mu(\{x\in X \mid |f(x)- m_{f}|\geq r\})\leq \beta(r)$ for any 1-Lipschitz function $f$ and any $r>0$, and $\bar{\beta}:=\int_{0}^{\infty}\beta(r)dr<\infty$, then
\be
\mu (\{x\in X \mid |f(x)-\int_{X} f d\mu |\geq r+\bar{\beta}\} )\leq \beta(r). \label{3.1-2}
\ee
In particular, if $\beta(r)\leq C {\rm e}^{-c r^{p}}$ for any $r>0$ and some constants $c>0$, $C>0$ and $0<p<+\infty$,  then
\be
\mu (\{x\in X \mid |f(x)- \int_{X} f d\mu |\geq r \} )\leq C'{\rm e}^{-\kappa_{p} cr^{p}}, \label{3,1-3}
\ee
where $C':=\max\{C,1\}{\rm e}^{\mathfrak{c}_{p}^{p}C^{p}}$, $\mathfrak{c}_{p}:=\Gamma(\frac{1}{p}+1)$ and $\kappa_{p}:=\min\{1, 2^{1-p}\}$.
\eei
\end{prop}
\begin{proof}
(1) Take $A$ with $\mu(A)>0$ and fix arbitrarily a $r>0$. Consider $f(x)=\min\{d(A, x), r\}$, $x\in X$. Clearly, $f$ is a 1-Lipschitz function. Noticing that
$$
\int_{X} f d\mu =\int_{A} f d\mu + \int_{M\setminus A} f d\mu \leq r \mu (M\setminus A) = r(1-\mu(A))
$$
and
$$
\{x\in X \mid d(A, x)\geq r\}\subset \{x\in X \mid f(x)= r\},
$$
we have that
$$
1-\mu (B^{+}(A,r) )\leq \mu ( \{x\in X \mid f(x)= r \} )\leq \mu ( \{x\in X \mid f(x)\geq \int_{X} f d\mu+\mu(A)r \} ).
$$
By the hypothesis, we obtain that $1-\mu(B^{+}(A,r))\leq \beta (\mu(A) r)$.

Similarly, let $f(x)=\max\{-d(x, A), -r\}$, $x\in X$. It is easy to see that $f$ is also a 1-Lipschitz function and $-\int_{X} f d\mu\leq r(1-\mu(A))$. Then we obtain the following
$$
1-\mu (B^{-}(A,r) )\leq \mu ( \{x\in X \mid f(x)= -r \} )\leq \mu ( \{x\in X \mid \int_{X} f d\mu\geq f(x)+\mu(A)r \})\leq \beta \left(\mu(A) r\right).
$$

In particular, for any Borel set $A$ with $\mu(A)\geq \frac{1}{2}$, by (\ref{3.2-1}) and the definition of $\alpha_{\mu}(r)$, we have $\alpha_{\mu}(r)\leq \beta(r/2)$.
\vskip 2mm

(2) If $\mu(\{x\in X \mid |f(x)- m_{f}|\geq r\})\leq \beta(r)$ and $\bar{\beta}=\int_{0}^{\infty}\beta(r)dr<\infty$, then,   we have
\beq
\left|\int_{X} f d\mu - m_{f}\right| &\leq & \int_{X} |f-m_{f} | d\mu  =\int_{X} \int_{0}^{\infty} {\bf 1}_{\{x \in X \mid | f(x) -m_{f}|> r\}}(x) dr d \mu \nonumber\\
& = &\int_{0}^{\infty} \int_{X} {\bf 1}_{\{x \in X \mid  | f(x)-m_{f}| > r \}}(x) d \mu d r =\int_{0}^{\infty} \mu (\{x \in X \mid | f(x) -m_{f}|> r \}) dr \nonumber \\
& \leq & \bar{\beta}, \label{layercake}
\eeq
where ${\bf 1}_A$ denotes the characteristic function of set $A$.

From (\ref{layercake}), we know that $\{x\in X \mid |f(x)-\int_{X} f d\mu|\geq r+\bar{\beta}\}\subset \{x\in X \mid |f(x)- m_{f}|\geq r\}$. Then
\be
\mu ( \{x\in X \mid |f(x)-\int_{X} f d\mu |\geq r+\bar{\beta} \} )\leq \beta(r).\label{betar}
\ee

In particular, if $\beta(r)\leq C {\rm e}^{-c r^{p}}$ for any $r>0$ with $c>0$, $C>0$ and $0<p<+\infty$, then
$$
\bar{\beta}\leq C\int_{0}^{\infty} {\rm e}^{-c r^{p}} dr=Cc^{-\frac{1}{p}}\Gamma(\frac{1}{p}+1)=:Cc^{-\frac{1}{p}}\mathfrak{c}_{p},
$$
where $\mathfrak{c}_{p}=\Gamma(\frac{1}{p}+1)$.

On the one hand, when $r\leq \bar{\beta}$, we have from $\bar{\beta}\leq Cc^{-\frac{1}{p}}\mathfrak{c}_{p}$ that
\be
\mu ( \{x\in X \mid |f(x)-\int_{X} f d\mu |\geq r \} )\leq 1\leq {\rm e}^{\mathfrak{c}_{p}^{p}C^{p}}{\rm e}^{-c\bar{\beta}^{p}} \leq  {\rm e}^{\mathfrak{c}_{p}^{p}C^{p}}{\rm e}^{-cr^{p}}.\label{betar-1}
\ee
On the other hand, when $r>\bar{\beta}$, by letting $s=r+\bar{\beta}$ in (\ref{betar}), we have
$$
\mu ( \{x \in X \mid | f(x)-\int_{X} f d\mu | \geq s \} ) \leq \beta(s-\bar{\beta}).
$$
Replacing $s$ by $r$ yields
\be
\mu ( \{ x \in X \mid | f(x)-\int_{X} f d\mu | \geq r \} ) \leq \beta(r-\bar{\beta}) \leq C e^{-c (r-\bar{\beta})^p}.  \label{betaba}
\ee
By using the inequality  $(a+b)^{p}\leq \max\{1, 2^{p-1}\}(a^{p}+b^{p})$ for $a, b\geq 0$, we have
$$
r^{p} \leq \max \{1,2^{p-1} \}\left((r-\bar{\beta})^{p}+\bar{\beta}^p\right),
$$
from which, we obtain that $(r-\bar{\beta})^{p} \geq \min \left\{1,2^{1-p}\right\} r^{p}-\bar{\beta}^p$. Thus we can get from (\ref{betaba}) that
\be
\mu ( \{x\in X \mid |f(x)-\int_{X} f d\mu |\geq r \} )  \leq  C{\rm e}^{c \bar{\beta}^{p}}{\rm e}^{-c\min\{1, 2^{1-p}\}r^{p}} \leq  C{\rm e}^{\mathfrak{c}_{p}^{p}C^{p}}{\rm e}^{-c\min\{1, 2^{1-p}\}r^{p}}.\label{betar-2}
\ee
 Combining (\ref{betar-1}) with (\ref{betar-2}) yields
$$
\mu ( \{x\in X \mid |f(x)-\int_{X} f d\mu |\geq r \} )\leq C' {\rm e}^{-c\kappa_{p}r^{p}},
$$
where $C':=\max\{C,1\}{\rm e}^{\mathfrak{c}_{p}^{p}C^{p}}$ and $\kappa_{p}:=\min\{1, 2^{1-p}\}$. This completes the proof.
\end{proof}

From Proposition \ref{f-mean}, we can prove the following theorem.

\begin{thm}\label{cor-nor}
Let $(X, d, \mu)$ be an irreversible metric measure space. Then $(X, d, \mu)$ has normal concentration (\ref{normal}) if and only if for every 1-Lipschitz function $f$ on $(X, d, \mu)$,
\be
\mu(\{x\in X \mid |f(x)- \int_{X} f d\mu|\geq r\})\leq C'{\rm e}^{-c'r^{2}}, \label{th3.3}
\ee
where $C'= C'(C)$ and $c'=c'(c)$ are constants.
\end{thm}
\begin{proof}
If $\alpha_{(X, d, \mu)}(r)\leq C {\rm e}^{-cr^2}$ for any $r>0$, by (\ref{mf3}), $\mu\left(\left\{z \in X \mid |f(z)-m_{f}| \geq r\right\}\right) \leq 2 \alpha_{\mu}(r) \leq 2 C e^{-c r^{2}}$. Then, by Proposition \ref{f-mean}(2) with $p=2$ and $\beta(r) =2 C e^{-c r^2}$, we have the desired result, i.e. (\ref{th3.3}). In this case, $C' =\max \{2C, 1\} e^{4\Gamma (\frac{3}{2})^{2}C^{2}}$ and $c'=\frac{1}{2}c$.

Conversely, assume that (\ref{th3.3}) holds. Then, from Proposition \ref{f-mean} (1) with  $\beta (r)= C' e^{-c' r^2}$, we can derive $\alpha_{\mu}(r) \leq \beta(r/2)= C e^{-cr^{2}}$, where $C =C', c=\frac{1}{4}c'$.
\end{proof}

Similarly, we can also prove the following theorem from Proposition \ref{f-mean}.

\begin{thm}\label{cor-ex}
Let $(X, d, \mu)$ be an irreversible metric measure space. Then $(X, d, \mu)$ has exponential concentration (\ref{exp}) if and only if for every 1-Lipschitz function $f$ on $(X, d, \mu)$,
\be
\mu(\{x\in X \mid |f(x)- \int_{X} f d\mu|\geq r\})\leq C'{\rm e}^{-c'r},
\ee
where $C'= C'(C)$ and $c'=c'(c)$ are constants.
\end{thm}

In \cite{EMil}, E. Milman introduced the definitions of the First-Moment concentration and linear tail-decay on reversible metric measure spaces. In the following, inspired by \cite{EMil}, we define the $(p, q)$-Moment concentration and linear tail-decay on an irreversible metric measure space $(X, d, \mu)$. Furthermore, we explore the relationship between the $(p, q)$-Moment concentration, linear tail-decay, the normal concentration, and the exponential concentration.
\begin{Def}
Given $p, q\geq 1$.  The space $(X, d, \mu)$ is said to satisfy $(p, q)$-Moment concentration if there is a positive constant $C$ such that for any  1-Lipschitz function $f$,
\be
\|f-\int_{X}f d\mu \|_{L^{q}(\mu)}^{p}\leq \frac{q}{C}. \label{pqmo}
\ee
\end{Def}
Actually, the $(1,1)$-Moment concentration is just the generalization of the First-Moment concentration  on reversible metric measure spaces defined by E. Milman \cite{EMil}.
\begin{Def}\label{tail}
The space $(X, d, \mu)$ is said to satisfy linear tail-decay if there is a positive constant $C$ such that for any $r>0$ and any 1-Lipschitz function $f$,
\be
\mu ( \{x\in X \mid |f(x)-\int_{X} f d\mu | \geq r \} )\leq\frac{1}{Cr}. \label{Ltd}
\ee
\end{Def}

\begin{thm}\label{alpha-mean}
The normal concentration implies $(2, q)$-Moment concentration on an irreversible metric measure space  $(X, d, \mu)$.
\end{thm}
\begin{proof}
Assume that $\alpha_{(X, d, \mu)}(r)\leq C {\rm e}^{-cr^2}$ for some positive constants $c$ and $C$ and any $r>0$. Then for every 1-Lipschitz  function $f$ and every $r>0$, we have from Theorem \ref{cor-nor} that
$$
\mu(\{x\in X \mid |f(x)- \int_{X} f d\mu|\geq r\})\leq C'{\rm e}^{- c'r^{2}},
$$
It follows  that
\beqn
\|f-\int_{X} f d \mu \|_{L^{q}(\mu)}^{q} & = &\int_{X} \int_{0}^{\infty} q r^{q-1} {\bf 1}_{ \{x \in X \mid | f(x)-\int_{X} f d \mu | >r \}}(x) d r d \mu \\
& = &\int_{0}^{\infty} \int_{X} q r^{q-1} {\bf 1}_{ \{x \in X \mid | f(x)-\int_{X} f d \mu | > r \}}(x) d \mu d r \\
& =& \int_{0}^{\infty} q r^{q-1} \mu ( \{ x \in X \mid | f(x)-\int_{X} f d \mu |>r \} ) d r \\
& \leq & q C' \int_{0}^{\infty} r^{q-1} e^{-c' r^2} d r = c' \left(\frac{1}{c'}\right)^{\frac{q}{2}} \Gamma\left(\frac{q}{2}+1\right),
\eeqn
where ${\bf 1}_A$ denotes the characteristic function of set $A$. Because $\Gamma\left(\frac{q}{2}+1\right)$ is a strictly increasing function for $q\geq 1$, by the Stirling formula of Gamma function, that is, $\Gamma(\frac{q}{2}+1)\sim \sqrt{2\pi} (q/2)^{q/2+1/2} {\rm e}^{-q/2}$ ($q\rightarrow\infty$), we have the following
$$
\|f-\int_{X} f d\mu \|_{L^{q}(\mu)}^{q}\leq \sqrt{2\pi} C'\left(\frac{q}{c'}\right)^{\frac{q}{2}} \left(\frac{q}{2}\right)^{\frac{1}{2}} 2^{-q/2} e^{-q/2}.
$$
Then
\be
\|f-\int_{X} f d\mu \|_{L^{q}(\mu)}\leq \sqrt{2\pi}C' {\rm e}^{1/(4{\rm e})-1/2}\sqrt{\frac{q}{c'}}  \label{nomo}
\ee
for any $q\geq 1$, where we have used that $\left(\frac{q}{2}\right)^{1/(2q)}\leq {\rm e}^{1/(4{\rm e})}$.  From (\ref{nomo}) we can conclude that the space $(X, d, \mu)$ satisfies $(2, q)$-Moment concentration.
\end{proof}

Conversely, we have the following theorem.

\begin{thm}\label{Moment-nol}
Let $(X, d, \mu)$ be an irreversible metric measure space. If $(X, d, \mu)$ satisfies $(2, q)$-Moment concentration (\ref{pqmo}) for $q\geq 1$ , then $(X, d, \mu)$ satisfies linear tail-decay when $0<r < \sqrt{\frac{\rm e}{C}}$ and  has normal concentration when $r> \sqrt{\frac{\rm e}{C}}$.
\end{thm}

\begin{proof}
By the Chebyshev inequality, for every $r>0$ and $q\geq 1$, we have
$$
\mu (\{ x\in X \mid |f(x)- \int_{X} f d\mu | \geq r \} )\leq r^{-q} \int_{X} |f(x)-\int_{X} f d\mu |^{q} d\mu \leq r^{-q} \left(\frac{q}{C}\right)^{q/2}.
$$
Let $h(q)=r^{-q} \left(\frac{q}{C}\right)^{q/2}$. Then  $\min\limits_{q\geq 1} h(q)=h(\frac{Cr^{2}}{{\rm e}})=\exp(-\frac{Cr^{2}}{2{\rm e}})$ when $r^{2}> \frac{{\rm e}}{C}$,  while $\min\limits_{q\geq 1} h(q)=h(1)= \left(\frac{1}{C}\right)^{1/2}r^{-1}$ when $r^{2}< \frac{{\rm e}}{C}$.
Thus, from Definition \ref{tail} and Theorem \ref{cor-nor}, we can derive our desired results respectively.
\end{proof}

Actually, by the Chebyshev inequality, we can also obtain the following result.
\begin{thm}\label{Moment-p1}
$(p, 1)$-Moment concentration implies linear tail-decay on an irreversible metric measure space  $(X, d, \mu)$.
\end{thm}
\begin{proof} By the assumption, we have
$$
\|f-\int_{X} f d \mu \|_{L^{1}(\mu)}^{p} \leq \frac{1}{C}, \  \text{ i.e.,} \ \left(\int_{X} |f-\int_{X} f d \mu | d \mu \right)^{p} \leq \frac{1}{C}.
$$
By the Chebyshev inequality,
$$
\mu (\{x \in X \mid |f(x)-\int_{X} f d\mu | \geq r \} ) \leq r^{-1} \int_{X} |f-\int_{X} f d \mu | d \mu \leq r^{-1} \cdot\frac{1}{C^{\frac{1}{p}}}=\frac{1}{C^{\frac{1}{p}} r} .
$$
This completes the proof of the theorem.
\end{proof}

\vskip 2mm

By the similar discussions, we have the following results regarding exponential concentration.

\begin{thm}\label{alpha-ex}
The exponential concentration implies $(1, q)$-Moment concentration on an irreversible metric measure space  $(X, d, \mu)$.
\end{thm}

\begin{thm}\label{Moment-exl}
Let $(X, d, \mu)$ be an irreversible metric measure space. If $(X, d, \mu)$ satisfies $(1, q)$-Moment concentration  (\ref{pqmo}) for $q\geq 1$, then $(X, d, \mu)$ satisfies linear tail-decay when $0<r < \frac{\rm e}{C}$ and has exponential concentration when $r> \frac{\rm e}{C}$.
\end{thm}
\vskip 2mm

Finsler metric measure manifolds form a special class of irreversible metric measure spaces.  Hence all discussions as above hold on Finsler metric measure manifolds. In particular, the following results corresponding respectively to Theorem {\ref{cor-nor}}, Theorem {\ref{cor-ex}}, Theorem {\ref{alpha-mean}} and Theorem {\ref{alpha-ex}} hold on Finsler metric measure manifolds.

\begin{thm}
Let $(M, F, \mathfrak{m})$ be a Finsler metric measure manifold with $\mathfrak{m}(M)=1$. Then $(M, F, \mathfrak{m})$ has normal concentration (\ref{normal}) if and only if for every 1-Lipschitz function $f$ on $(M, F, \mathfrak{m})$,
$$
\mathfrak{m}(\{x\in M \mid |f(x)- \int_{M} f d\mathfrak{m}|\geq r\})\leq C'{\rm e}^{-c'r^{2}},
$$
where $C'= C'(C)$ and $c'=c'(c)$ are constants.
\end{thm}

\begin{thm}
Let $(M, F, \mathfrak{m})$ be a Finsler metric measure manifold with $\mathfrak{m}(M)=1$. Then $(M, F, \mathfrak{m})$ has exponential concentration (\ref{exp}) if and only if for every 1-Lipschitz function $f$ on $(M, F, \mathfrak{m})$,
$$
\mathfrak{m}(\{x\in M \mid |f(x)- \int_{M} f d\mathfrak{m}|\geq r\})\leq C'{\rm e}^{-c'r},
$$
where $C'= C'(C)$ and $c'=c'(c)$ are constants.
\end{thm}

\begin{thm}
The normal concentration implies $(2, q)$-Moment concentration on a Finsler metric measure manifold  $(M, F, \mathfrak{m})$ with $\mathfrak{m}(M)=1$.
\end{thm}

\begin{thm}
The exponential concentration implies $(1, q)$-Moment concentration on a Finsler metric measure manifold  $(M, F, \mathfrak{m})$ with $\mathfrak{m}(M)=1$.
\end{thm}

In \cite{EMil},  E. Milman proved the equivalence of the First-Moment concentration inequality, the exponential concentration inequality, Cheeger's isoperimetric inequality and Poincar\'{e}'s inequality on Riemannian manifolds under the convexity assumptions.
It is an important and interesting topic to consider such  equivalence on Finsler metric measure manifolds.

\section{Observable diameter and concentration}\label{diacon}

In 1999, M. Gromov \cite{GROMOV} introduced the notion of observable diameter on reversible metric measure spaces. Later, Z. Shen \cite{SHEN} introduced the notion of  observable diameter on Finsler metric measure manifolds.
In this section, following the definition in \cite{SHEN}, we introduce the observable diameter ${\rm ObsDiam}_{\mu}(X,\kappa)$ on irreversible metric measure spaces $(X, d, \mu)$. Concretely, for $0<\kappa<1$, define first the partial diameter
\[
{\rm ParDiam}_{\mu}(X; \kappa):=\inf\limits_{A}\{{\rm diam}(A) \mid A\subset X, \ \mu(A)\geq 1-\kappa\}, \label{fpdia}
\]
where ${\rm diam}(A):=\sup_{x, z\in A}d(x, z)$. Then we define the observable diameter as
\beq
{\rm ObsDiam}_{\mu}(X; \kappa)&:=& \sup\limits_{f\in {\rm Lip}_{1}(X)}{\rm ParDiam}_{f_{*}\mu}(\mathbb{R}; \kappa) \nonumber\\
&=& \sup\limits_{f\in {\rm Lip}_{1}(X)} \inf\limits_{I}\{{\rm diam}(I) \mid I\subset \mathbb{R}, \ \mu(f^{-1}(I))\geq 1-\kappa\}, \label{obdia}
\eeq
where $f_{*}\mu$ is the push-forward of $\mu$ by the function $f$.

The following theorem establishes the relationship between the observable diameter ${\rm ObsDiam}_{\mu}(X; \kappa)$ and the concentration function $\alpha_{\mu}(r)=\alpha_{(X, d, \mu)}(r)$. Let $\alpha_{\mu}^{-1}(\varepsilon)$ be the generalized inverse function of the non-increasing function $\alpha_{\mu}(r)$, that is,
\be
\alpha_{\mu}^{-1}(\varepsilon):=\inf \{r>0 \mid \alpha_{\mu}(r)\leq \varepsilon\} \label{inverse}
\ee
for $\varepsilon>0.$

\begin{thm}\label{osb-alpha}
Let $(X, d, \mu)$ be an irreversible metric measure space. Then for every $0<\varepsilon<1$,
\be
{\rm ObsDiam}_{\mu}(X; \varepsilon)\leq 2\alpha_{\mu}^{-1}\left(\frac{\varepsilon}{2}\right).  \label{esObs}
\ee
\end{thm}
\begin{proof}
For any $f\in {\rm Lip}_{1}(X)$, we have from inequality (\ref{mf3}) that for every $r>0$,
$$
\mu(\{x\in X \mid |f(x)- m_{f}|\geq r\})\leq 2\alpha_{\mu}(r),
$$
and then,
\[
 \mu(\{x\in X \mid f(x)\in(m_{f}-r, m_{f}+r)\})\geq 1-2\alpha_{\mu}(r).
\]
Further, for any $r>0$ satisfying $\alpha_{\mu}(r)\leq \frac{\varepsilon}{2}$, we have
$$
\mu(\{x\in X \mid f(x)\in(m_{f}-r, m_{f}+r)\})\geq 1-\varepsilon.
$$
Because the interval $[m_{f}-r, m_{f}+r]$ has length $2r$, from the definition of ${\rm ParDiam}_{f_{*}\mu}(\mathbb{R}; \varepsilon)$, we have
$$
{\rm ParDiam}_{f_{*}\mu}(\mathbb{R}; \varepsilon)\leq 2r.
$$
From this,  we first take the infimum over all $r>0$ satisfying $\alpha_{\mu}(r)\leq \frac{\varepsilon}{2}$,  and then, by using (\ref{inverse}) we can obtain that
$$
{\rm ParDiam}_{f_{*}\mu}(\mathbb{R}; \varepsilon)\leq 2\alpha_{\mu}^{-1}\left(\frac{\varepsilon}{2}\right).
$$
Thus, by the definition of ${\rm ObsDiam}_{\mu}(X; \varepsilon)$ and taking the supremum over all 1-Lipchisz function $f$ in the above inequality,  we can conclude the following
$$
{\rm ObsDiam}_{\mu}(X; \varepsilon)\leq 2\alpha_{\mu}^{-1}\left(\frac{\varepsilon}{2}\right).
$$
This completes the proof.
\end{proof}

Note that the inverse function $\alpha_{\mu}^{-1}(\varepsilon)$ is closely related to the expansion distance ${\rm ExDist(M, \varepsilon)}$ defined by Z. Shen \cite{SHEN} in Finsler setting. Hence (\ref{esObs}) is an analogue of  inequality (4.29) of \cite{SHEN} on Finsler metric measure manifolds $(M, F, \mathfrak{m})$ with $\mathfrak{m}(M)=1$.

According to the Theorem \ref{osb-alpha}, if $(X, d, \mu)$ has normal concentration (\ref{normal}), then  for $0<\varepsilon<1$,
\be
{\rm ObsDiam}_{\mu}(X; \varepsilon)\leq 2\sqrt{\frac{1}{c}\log\frac{2C}{\varepsilon}}.\label{Ob-nor}
\ee
If $(X, d, \mu)$ has exponential concentration (\ref{exp}), then for $0<\varepsilon<1$,
\be
{\rm ObsDiam}_{\mu}(X; \varepsilon)\leq \frac{2}{c}\log\frac{2C}{\varepsilon}.\label{Ob-ex}
\ee

\section{Isoperimetric inequality and concentration}\label{Iso}
Let $(M, F, \mathfrak{m})$ be an $n$-dimensional Finsler manifold $(M, F)$ equipped with a positive smooth measure $\mathfrak{m}$. Given a subset $E\subset M$, the forward and backward Minkowski exterior boundary measure (or Minkowski content) are respectively defined as
\be
\mathfrak{m}^{+}(E):=\liminf\limits_{\varepsilon\rightarrow 0^{+}} \frac{\mathfrak{m}(B^{+}(E, \varepsilon))-\mathfrak{m}(E)}{\varepsilon} \ \ \text{and} \ \ \mathfrak{m}^{-}(E):=\liminf\limits_{\varepsilon\rightarrow 0^{+}} \frac{\mathfrak{m}(B^{-}(E, \varepsilon))-\mathfrak{m}(E)}{\varepsilon}. \label{Mcontent}
\ee
The isoperimetric profile $I_{\mathfrak{m}}: [0, \mathfrak{m}(M)]\rightarrow [0, \infty)$ of $(M, F, \mathfrak{m})$ is defined by
\be
I_{\mathfrak{m}}(\mathfrak{m}(E)):=\inf\{\min\{\mathfrak{m}^{+}(E), \mathfrak{m}^{-}(E)\} \ | \ E\subset M: \text{Borel set with \ } \mathfrak{m}^{+}(E)<\infty \ \text{and} \ \mathfrak{m}^{-}(E)<\infty \}. \label{isopro}
\ee

In the following, similar to the arguments in Chapter 2 of \cite{Le},  we always regard $\liminf$ in (\ref{Mcontent}) as the limit for Borel subsets.
\begin{lem}\label{iso-I}
Let $(M, F, \mathfrak{m})$ be an $n$-dimensional Finsler metric measure manifold. Assume that $I_{\mathfrak{m}}\geq \varphi'\circ \varphi^{-1}$ for some strictly increasing differentiable function $\varphi: I\subset \mathbb{R}\rightarrow[0, \mathfrak{m}(M)]$. Then, for every $r>0$,
$$
\min\left\{\varphi^{-1}\left(\mathfrak{m}(B^{+}(E, r))\right), \varphi^{-1}\left(\mathfrak{m}(B^{-}(E, r))\right)\right\}\geq \varphi^{-1}\left(\mathfrak{m}(E)\right)+r.
$$
\end{lem}
\begin{proof}
Let $h(r)=\varphi^{-1}(\mathfrak{m}(B^{+}(E, r)))$, then by the hypothesis and the definition of the isoperimetric profile $I_{\mathfrak{m}}$, we have that
$$
h'(r)=\frac{\mathfrak{m}^{+}(B^{+}(E, r))}{\varphi'\circ \varphi^{-1}(\mathfrak{m}(B^{+}(E, r)))} \geq 1.
$$
Thus $h(r)=h(0)+\int_{0}^{r}h'(s)ds\geq \varphi^{-1}(\mathfrak{m}(E))+r$, that is, $\varphi^{-1}(\mathfrak{m}(B^{+}(E, r))) \geq \varphi^{-1}(\mathfrak{m}(E))+r$.

Similarly, for the backward Minkowski content we can also obtain that $\varphi^{-1}(\mathfrak{m}(B^{-}(E, r)))\geq \varphi^{-1}(\mathfrak{m}(E))+r$. Thus, our desired result holds.
\end{proof}

\begin{lem}\label{iso-phi}
Let $(M, F, \mathfrak{m})$ be an $n$-dimensional Finsler metric measure manifold with $\mathfrak{m}(M)=1$. Assume that $I_{\mathfrak{m}}\geq \varphi'\circ \varphi^{-1}$ for some strictly increasing differentiable function $\varphi: I\subset \mathbb{R}\rightarrow[0, 1]$. Then, for every $r>0$,
$$
\alpha_{(M, F, \mathfrak{m})}(r)\leq 1-\varphi(\varphi^{-1}(1/2)+r).
$$
\end{lem}
\begin{proof}
Since $\varphi$ is a strictly increasing function, by Lemma \ref{iso-I} we have
$$
\min\{\mathfrak{m}(B^{+}(E, r)), \mathfrak{m}(B^{-}(E, r))\}\geq \varphi\left(\varphi^{-1}(\mathfrak{m}(E))+r\right).
$$
For any $E\subset M$ satisfying $\mathfrak{m}(E)\geq \frac{1}{2}$, it is not difficult to see that  $\varphi^{-1}(\mathfrak{m}(E))\geq \varphi^{-1}(1/2)$ and $
\alpha_{(M, F, \mathfrak{m})}(r)\leq 1-\varphi(\varphi^{-1}(1/2)+r)$ by the definition of the concentration function $\alpha_{(M, F, \mathfrak{m})}(r)$.
\end{proof}

In \cite{OHTA}, S. Ohta  proved a Bakry-Ledoux type isoperimetric inequality on compact Finsler metric measure manifolds satisfying $\mathfrak{m}(M)=1$ and ${\rm Ric}_{\infty}\geq K>0$ (\cite{OHTA}, Theorem 15.1. Also see \cite{ohtabl}). Although the isoperimetric profile $I_{\mathfrak{m}}(\mathfrak{m}(E))$ in this paper is slightly different from the isoperimetric profile $I_{(M, F, \mathfrak{m})}(\mathfrak{m}(E))$ defined by S. Ohta in \cite{OHTA}, we can prove the following lemma by following Ohta's argument.
\begin{lem} \label{O-iso}
Assume that $(M, F, \mathfrak{m})$ is compact and satisfies $\mathfrak{m}(M)=1$ and ${\rm Ric}_{\infty}\geq K>0$. Then we have
$$
I_{\mathfrak{m}}(\mathfrak{m}(E))\geq \sqrt{K}\varphi'\circ \varphi^{-1}(\mathfrak{m}(E))
$$
for all $E\subset M$, where $\varphi(t):=\frac{1}{\sqrt{2\pi}}\int_{-\infty}^{t}{\rm e}^{-b^{2}/2}d b$, $t=\frac{\varphi^{-1}(\mathfrak{m}(E))}{\sqrt{K}}$ and $\varphi'\circ \varphi^{-1}(\mathfrak{m}(E))=\frac{1}{\sqrt{2\pi}}{\rm e}^{-Kt^{2}/2}$.
\end{lem}

By Lemma \ref{iso-phi} and Lemma \ref{O-iso}, we have
\begin{thm}\label{iso-nor}
Let $(M, F, \mathfrak{m})$ be a compact Finsler metric measure manifold satisfying $\mathfrak{m}(M)=1$ and ${\rm Ric}_{\infty}\geq K>0$. Then $(M, F, \mathfrak{m})$ has normal concentration.
\end{thm}
\begin{proof}
 Let  $\phi(t)=\varphi(\sqrt{K}t)=\frac{1}{\sqrt{2\pi}}\int_{-\infty}^{\sqrt{K}t}{\rm e}^{-b^{2}/2}d b$.  From Lemma \ref{O-iso},  we have
$$
\phi(t)=\varphi(\sqrt{K}t)=\mathfrak{m}(E), \ \phi'\circ \phi^{-1}(\mathfrak{m}(E))=\sqrt{K}\varphi'\circ \varphi^{-1}(\mathfrak{m}(E))=\sqrt{\frac{K}{2\pi}}{\rm e}^{-Kt^{2}/2}
$$
and $\phi^{-1}(1/2)=0$. Then, by Lemma \ref{iso-phi}, we have
$$
\alpha_{(M, F, \mathfrak{m})}(r)\leq 1-\phi(r)=\frac{1}{\sqrt{2\pi}}\int^{\infty}_{\sqrt{K}r}{\rm e}^{-b^{2}/2}d b\leq \frac{1}{2} {\rm e}^{-\frac{K}{2} r^{2}}
$$
for $r>0$, where we have used the fact that $\int_{t}^{\infty}{\rm e}^{-\tau^{2}/2}d\tau\leq \sqrt{\frac{\pi}{2}}{\rm e}^{-t^{2}/2}$ for $t\geq 0$. This completes the proof.
\end{proof}

From Theorem \ref{iso-nor} and by inequality (\ref{Ob-nor}), we can obtain the following result.
\begin{cor}
Let $(M, F, \mathfrak{m})$ be a compact Finsler metric measure manifold satisfying $\mathfrak{m}(M)=1$ and ${\rm Ric}_{\infty}\geq K>0$. Then for $0<\varepsilon<1$,
$$
{\rm ObsDiam}_{\mathfrak{m}}(M; \varepsilon)\leq 2\sqrt{\frac{2}{K}\log\frac{1}{\varepsilon}}.
$$
\end{cor}

\section{Spectrum and concentration}\label{Spec}

In this section, following the arguments of Gromov-Milman \cite{Gr} and M. Ledoux \cite{Le}, we will prove an exponential concentration of closed Finsler metric measure manifolds via the first eigenvalue.  Further, we will derive a Cheng type upper bound estimate for the first closed eigenvalue by the concentration properties.

Let $(M, F, \mathfrak{m})$ be a closed Finsler metric measure manifold, the first positive eigenvalue is defined by (see \cite{SHEN})
$$
\lambda_1(M):=\inf _{u\in C^{\infty}(M)} \frac{\int_{M}\left(F^{*}(d u)\right)^2 d \mathfrak{m}}{\inf\limits_{\lambda\in \mathbb{R}}\int_{M} |u-\lambda|^2 d \mathfrak{m}}.
$$

Firstly, we prove the following theorem which gives a upper bound estimate for the concentration function controlled by the first eigenvalue on closed Finsler metric measure manifolds, which shows an exponential concentration property.

\begin{thm}\label{con-L}
Let $(M, F, \mathfrak{m})$ be a closed Finsler metric measure manifold with $\mathfrak{m}(M)=1$. Then $(M, F, \mathfrak{m})$ has exponential concentration,
\begin{equation}\label{lambda-con}
\alpha_{(M, F, \mathfrak{m})}(r)\leq {\rm e}^{-r\sqrt{\lambda_1(M)}(\log 2/\sqrt{2})},\ r>0.
\end{equation}
\end{thm}
\begin{proof}
Let $A$ and $B$ be measurable subsets of $M$ such that $d_{F}(B, A)=\rho>0$. Set $\mathfrak{m}(A)=a>0$ and $\mathfrak{m}(B)= b>0$. Consider  function
$$
f(x)=\frac{1}{a}-\frac{1}{\rho}\left(\frac{1}{a}+\frac{1}{b}\right) \min\{d_{F}(x, A), \rho\} .
$$
Notice that $f=\frac{1}{a}$ on $A$ and $f=-\frac{1}{b}$ on $B$. Hence we have $\nabla f=0$ on $A\cup B$ and $F(\nabla f)\leq \frac{1}{\rho}\left(\frac{1}{a}+\frac{1}{b}\right)$ almost everywhere on $M\setminus(A\cup B)$.  Moreover, it follows from the triangle inequality that $f$ is a Lipschitz function. Therefore, for any $\lambda\in \mathbb{R}$,
$$
\lambda_{1}(M) \inf_{\lambda\in \mathbb{R}}\int_{M} |f-\lambda|^{2} d\mathfrak{m}\leq\int_{M}F^{2}(\nabla f) d\mathfrak{m}\leq \frac{1}{\rho^2}\left(\frac{1}{a}+\frac{1}{b}\right)^2(1-a-b)
$$
and
\beqn
\lambda_1(M)\int_{M} |f-\lambda|^{2} d\mathfrak{m} &\geq& \lambda_1(M)\left(\int_{A}\left(\frac{1}{a}-\lambda\right)^{2} d\mathfrak{m}+\int_{B}\left(\frac{1}{b}+\lambda\right)^{2} d\mathfrak{m}\right)\\ &=&\lambda_1(M)\left(\left(\frac{1}{a}-\lambda\right)^{2}a+ \left(\frac{1}{b}+\lambda\right)^{2}b\right)\\ &\geq& \lambda_1(M)\left(\frac{1}{a}+\frac{1}{b}\right).
\eeqn
Then $\rho^2 \lambda_1(M) \leq\left(\frac{1}{a}+\frac{1}{b}\right)(1-a-b)\leq \frac{1-a-b}{ab}$, i.e., $b \leq \frac{1-a}{1+\rho^2\lambda_1(M) a}$.

Now, for any $\varepsilon>0$, let us consider a sequence of pairs $(A_{i}, B_{i})$ of subsets of $M$, such that  $A_0=A$, $A_{1}=B^{-}(A, \varepsilon)$, $A_{2}=B^{-}(A_{1}, \varepsilon)$, $\cdots$, $A_{k}=B^{-}(A_{k-1}, \varepsilon)$ and $B_0=M \backslash A_{1}$, $B_1=M\backslash A_{2}$, $\cdots$, $B_{k}=M\backslash A_{k+1}$, $k=0, 1, 2, \cdots$. By the similar discussions as above,  after $k$-steps we have
$$
b_{k} \leq \frac{1-a_{k}}{1+\lambda_1(M) \varepsilon^{2} a_{k}},
$$
where $b_{k}=\mathfrak{m}\left(B_{k}\right), a_{k}=\mathfrak{m}\left(A_{k}\right)$ and $a_{k+1}=1-b_{k}$.

In the following, take $\varepsilon$ such that $\lambda_1(M)\varepsilon^{2}=2$. For any $A\subset M$ satisfying $\mathfrak{m}(A)=a\geq 1/2$, $a_{k}\geq 1/2$ and
$$
1-a_{k+1} =b_{k}\leq \frac{1-a_{k}}{1+\lambda_1(M) \varepsilon^{2}/2} = \frac{1}{2} \left(1-a_{k}\right).
$$
By iteration, we obtain
$$
1-a_{k} \leq \frac{1}{2^{k}}\left(1-a\right)\leq \frac{1}{2^{k+1}}= {\rm e}^{-(k+1)\log 2}.
$$
Since $A_{k}\subset B^{-}(A, k\varepsilon)$, we have $1-\mathfrak{m}(B^{-}(A, k\varepsilon))\leq {\rm e}^{-(k+1)\log 2}$.

For any $r>0$, we can choose $k$ such that $k\varepsilon< r\leq (k+1)\varepsilon$.  Then
$$
1-\mathfrak{m}(B^{-}(A, r))\leq {\rm e}^{-r\sqrt{\lambda_1(M)}\log 2/\varepsilon}= {\rm e}^{-r\sqrt{\lambda_1(M)}(\log 2/\sqrt{2})}.
$$

Similarly, we can also prove that $1-\mathfrak{m}(B^{+}(A, r))\leq {\rm e}^{-r\sqrt{\lambda_1(M)}(\log 2/\sqrt{2})}$ by replacing $F$ with $\overleftarrow{F}$. This completes the proof.
\end{proof}

From Theorem \ref{con-L} and by inequality (\ref{Ob-ex}), we have the following
\begin{cor}
Let $(M, F, \mathfrak{m})$ be a closed Finsler metric measure manifold with $\mathfrak{m}(M)=1$. Then for $0<\varepsilon<1$,
$$
{\rm ObsDiam}_{\mathfrak{m}}(M; \varepsilon)\leq \frac{2\sqrt{2}}{\log 2}\log\frac{2}{\varepsilon}\cdot\frac{1}{\sqrt{\lambda_{1}(M)}}.
$$
\end{cor}

In Riemannian geometry, S. Cheng \cite{cheng} discovered the first Dirichlet eigenvalue comparison of a geodesic ball and further gave the upper bounds of the eigenvalues $\mu_{m}(M)$ $(0=\mu_{0}(M)<\mu_{1}(M)\leq \mu_{2}(M)\leq \cdots)$ in terms of the diameter of $M$ and the lower bound on Ricci curvature. Later, B. Chen \cite{chen} and Z. Shen \cite{shen2} extended Cheng type upper bound estimates for the first eigenvalue to the Finsler manifolds under different conditions, respectively.

In the following, we derive a new Cheng type upper bound estimate for the first closed eigenvalue as an interesting application of Theorem \ref{con-L}.
\begin{thm}
Let $(M, F, \mathfrak{m})$ be an $n$-dimensional closed Finsler metric measure manifold with $\mathfrak{m}(M)=1$. Assume that {\rm Ric}$_{\infty}\geq K$ and  $|\tau|\leq a$ for some $K\in \mathbb{R}$ and $a\geq 0$.  Let $D$ be the diameter of $M$. Then
$$
\lambda_{1}(M)\leq \frac{C(n, a, K, D)}{D^{2}},
$$
where $C(n, a, K, D)$ is a constant depending on $n$, $a$, $K$ and $D$. Furthermore, if {\rm Ric}$_{\infty}\geq 0$, then $\lambda_{1}(M)\leq \frac{C(n, a)}{D^{2}}$, where $C(n, a)$ is a constant depending on $n$ and $a$.
\end{thm}
\begin{proof}
By Theorem \ref{con-L}, for any measurable subset $A$ in $M$ such that $\mathfrak{m}(A)\geq \frac{1}{2}$, and for every $r>0$,
\begin{equation}\label{lambda-1}
1-\min\{\mathfrak{m}(B^{+}(A, r)), \mathfrak{m}(B^{-}(A, r))\}\leq {\rm e}^{-r\sqrt{\lambda_1(M)}(\log 2/\sqrt{2})}.
\end{equation}

\textbf{Case 1.} $\mathfrak{m}(B^{+}_{x}(D/4))\geq 1/2$ for any $x\in M$. In this case,  take $z$ such that $d(x, z)=\frac{D}{2}+\frac{D}{4}$. It is easy to see that $B^{+}(B^{+}_{x}(D/4), D/4)\subset B^{+}_{x}(D/2)$ and $B^{-}_{z}(D/4)\subset M \backslash B^{+}_{x}(D/2)$. On the other hand, by the volume comparison theorems given in \cite{chengF} (i.e. \cite{chengF}, Propositions 3.2 and 3.3), we can obtain
$$
\frac{\mathfrak{m}(B^{-}_{z}(D))}{\mathfrak{m}(B^{-}_{z}(D/4))}\leq 4^{n+4a}{\rm e}^{D(n+4a)\sqrt{|K|/(n-1)}}.
$$
Hence, we can get
\beq\label{lambda-2}
1-\mathfrak{m}(B^{+}(B^{+}_{x}(D/4), D/4))& \geq & 1-\mathfrak{m}(B^{+}_{x}(D/2)) \geq \mathfrak{m}(B^{-}_{z}(D/4)) \nonumber\\
&\geq& 4^{-(n+4a)}{\rm e}^{-D(n+4a)\sqrt{|K|/(n-1)}}.
\eeq
By (\ref{lambda-1}) with $A=B^{+}_{x}(D/4)$ and (\ref{lambda-2}), we have
$$
{\rm e}^{(D/4)\sqrt{\lambda_1(M)}(\log 2/\sqrt{2})}\leq 4^{(n+4a)}{\rm e}^{D(n+4a)\sqrt{|K|/(n-1)}},
$$
from which, we obtain the following
$$
\lambda_{1}(M)\leq \frac{C(n, a, K, D)}{D^{2}},
$$
where $C(n, a, K, D)=32(n+4a)^{2}(2+D\sqrt{|K|}/(\sqrt{n-1}\log 2))^{2}$.
\vskip 2mm
\textbf{Case 2.} $\mathfrak{m}(B^{+}_{x}(D/4))< 1/2$ for any $x\in M$. Firstly, we apply (\ref{lambda-1}) to $A=M\backslash B^{+}_{x}(D/4):=B^{+}_{x}(D/4)^{c}$. Then, because $B^{+}_{x}(D/8)\subset M\backslash B^{-}\left(B^{+}_{x}(D/4)^{c}, D/8\right)$ and by using volume comparison theorems same as above, we have the following
$$
1-\mathfrak{m}(B^{-}(B^{+}_{x}(D/4)^{c}, D/8))\geq \mathfrak{m}(B^{+}_{x}(D/8))\geq 8^{-(n+4a)}{\rm e}^{-D(n+4a)\sqrt{|K|/(n-1)}}.
$$
Then, by the similar argument as in Case 1, we can derive
$$
\lambda_{1}(M)\leq \frac{C(n, a, K, D)}{D^{2}},
$$
where $C(n, a, K, D)=128(n+4a)^{2}(3+D\sqrt{|K|}/(\sqrt{n-1}\log 2))^{2}$. This completes the proof.
\end{proof}

\end{document}